\documentclass{article} 

\usepackage{enumerate,enumitem}
\usepackage{amssymb}
\usepackage{amsmath}
\usepackage{amsthm}
\usepackage{color}
\usepackage[all,pdf,cmtip]{xy}
\usepackage{titlefoot}

\title{Transfunctions and their connections to Plans, Markov Operators and Optimal Transport \\ ~\\ Preprint}
\author{Jason Bentley$^\dag$ and Piotr Mikusi\'{n}ski\\
Department of Mathematics\\
University of Central Florida\\
Orlando, Florida, USA}

\newcommand{\mc}[1]{\mathcal{#1}}
\newcommand{\mb}[1]{\mathbb{#1}}

\newcommand{\eps}{\varepsilon}

\newtheorem{theorem}{Theorem}[section]
\newtheorem{lemma}[theorem]{Lemma}
\newtheorem{proposition}[theorem]{Proposition}

\theoremstyle{definition}
\newtheorem{definition}[theorem]{Definition}
\newtheorem{example}[theorem]{Example}

\begin{document}
\maketitle

\unmarkedfntext{$^\dag$ Corresponding author.}
\unmarkedfntext{2010 \textit{Mathematics Subject Classification.} Primary 28D05, 28A33; Secondary 28C99.}
\unmarkedfntext{\textit{Key words and phrases.} Generalized function, transfunction, Markov operator, plan, Radon adjoint.}

\section*{Abstract}
A transfunction is a function which maps between sets of finite measures on measurable spaces. In this paper we characterize transfunctions that correspond to Markov operators and to plans; such a transfunction will contain the ``instructions" common to several Markov operators and plans. We also define the adjoint of transfunctions in two settings and provide conditions for existence of adjoints. Finally, we develop approximations of identity in each setting and use them to approximate weakly-continuous transfunctions with simple transfunctions; one of these results can be applied to some optimal transport problems to approximate the optimal cost with simple Markov transfunctions.

\pagebreak

\section{Introduction}

Let $\mc{M}_X$ and $\mc{M}_Y$ be vector spaces of finite signed measures defined on measurable spaces $(X,\Sigma_X)$ and $(Y, \Sigma_Y)$, respectively. For finite positive measure $\mu$ and real-valued function $f \in \mc{L}^1(X,\mu)$, let $f \mu$ denote the measure $A \mapsto \int_A f ~d\mu$
and define $$
\mc{M}^{p,+}_\mu = \{ f \mu : f \in \mc{L}^p(X,\mu),~ f \ge 0\} \quad \text{and } \quad \mc{M}^p_\mu := \{f \mu : f \in \mc{L}^p(X,\mu)\},
$$
for $p \in [1,\infty]$. We define $\mc{M}_\mu$ to be the set of all signed measures absolutely continuous with respect to $\mu$. By the Radon-Nikodym Theorem, $\mc{M}_\mu = \mc{M}_\mu^1$. Similarly, we define $\mc{M}^{p,+}_\nu$ and $\mc{M}^p_\nu$.

A \textit{transfunction} is any function $\Phi : \mathcal{M}_X \to \mathcal{M}_Y$, \cite{Mikusinski}. Strongly $\sigma$-additive transfunctions are those which are linear and continuous with respect to total variation. We will sometimes call an operator between Banach spaces \emph{strongly $\sigma$-additive} if it is linear and norm-continuous.

Plans have applications for finding weak solutions for optimal transport problems, \cite{Villani}. Markov operators, defined in Section 2, have some similarities to stochastic matrices, \cite{Brown}. Plans and Markov operators have a bijective correspondence as described in \cite{Taylor} and in Section 2. We assign to any corresponding Markov operator/plan pair $(T,\kappa)$ with marginals $\mu,\nu$ a unique transfunction $\Phi: \mc{M}_\mu \to \mc{M}_\nu$ -- called a \emph{Markov transfunction}. However, each Markov transfunction corresponds to a family of Markov operators (resp. plans) which have different marginals but follow the same ``instructions''. $\Phi$, $T$, and $\kappa$ are related via the equalities
$$
\Phi(1_A \mu)(B) = \int_B T(1_A) ~d\nu = \kappa(A \times B),
$$ 
$$
\int_Y g ~d\Phi(f \mu) = \int_Y T(f) ~d(g \nu) = \int_{X \times Y} (f \otimes g) ~d\kappa 
$$
which hold for all $A\subseteq X$, $B\subseteq Y$, $f \in \mc{L}^\infty(X)$, and $g \in \mc{L}^\infty(Y)$. The first set of equalities, although simpler, imply the second set of equalities by strong $\sigma$-additivity of $\Phi$, bounded-linearity of $T$, and $\sigma$-additivity of $\kappa$.

In our investigation of transfunctions we are motivated by the theory developed for the Monge-Kantorovich transportation problems and their  far-reaching outcomes; see \cite{Ambrosio}, \cite{Campi}, \cite{Guerra}, and \cite{Villani}.

Let $\mc{F}_X$ and $\mc{F}_Y$ be spaces of measurable functions which are integrable by measures in $\mc{M}_X$ and $\mc{M}_Y$, respectively. If $\{\mc{F}_X, \mc{M}_X\}$ and $\{\mc{F}_Y, \mc{M}_Y\}$ are separating pairs with respect to integration as defined in Section 3, then we define the \emph{Radon adjoint} of $\Phi: \mc{M}_X \to \mc{M}_Y$ (if it exists) to be the unique linear bounded operator $\Phi^\ast: \mc{F}_Y \to \mc{F}_X$ such that $$\int_X \Phi^\ast(g) ~d\mu = \int_Y g ~d\Phi(\lambda)$$ for all $g \in \mc{F}_Y$ and $\lambda \in \mc{M}_X$. 

If $X$ and $Y$ are second-countable locally compact Hausdorff spaces, if $\mc{F}_X$ and $\mc{F}_Y$ are Banach spaces of bounded continuous functions (uniform norm) and if $\mc{M}_X$ and $\mc{M}_Y$ are Banach spaces of finite regular signed measures (total variation), then any strongly $\sigma$-additive weakly-continuous transfunction $\Phi : \mc{M}_X \to \mc{M}_Y$ has an adjoint $\Phi^\ast$ which is a linear, uniformly-continuous and bounded-pointwise-continuous operator (and vice versa) such that $||\Phi|| = ||\Phi^\ast||$. In future research, we wish to develop functional analysis on transfunctions, and adjoints may be utilized to this end. In contexts where operators on functions are more appropriate or preferable, the adjoint may prove crucial.

A simple transfunction $\Phi: \mc{M}_X \to \mc{M}_Y$ is one which has the form
$$\Phi(\lambda) := \sum_{i=1}^m \langle f_i, \lambda \rangle \rho_i$$
for $f_1, \dots, f_m \in \mc{F}_X$ and $\rho_1,\dots, \rho_m \in \mc{M}_Y$, where $\langle f_i, \lambda \rangle := \int_X f_i ~d\lambda$. Simple transfunctions are weakly-continuous and strongly $\sigma$-additive. When working with locally compact Polish (metric) spaces, simple Markov transfunctions have two advantages: they weakly approximate all Markov transfunctions, and a subclass of them can be utilized to approximate the optimal cost between two marginals with respect to a transport cost $c(x,y)$ that is bounded by $\alpha d(x,y)^p$ for constants $\alpha, p > 0$. 

In \cite{JBPM1} the notions of localization of a transfunction and the graph of a transfunction are introduced and studied. They give us an insight into which transfunctions arise from continuous functions or measurable functions or are close to such functions.

\section{Markov Transfunctions}

In this section, we describe a class of transfunctions in which each transfunction corresponds to a family of plans and a family of Markov operators. First, we introduce these concepts. All measurable or continuous functions shall be real-valued in this text. Note that the following definitions allow for all finite positive measures rather than all probability measures.

\begin{definition}
Let $\mu$ and $\nu$ be finite positive measures on $(X,\Sigma_X)$ and $(Y,\Sigma_Y)$ respectively with $||\mu||=||\nu||$. Let $\kappa$ be a finite positive measure on the product measurable space $(X\times Y, \Sigma_{X \times Y})$. We say that $\kappa$ is a \emph{plan with marginals $\mu$ and $\nu$} if $\kappa(A \times Y) = \mu(A)$ and $\kappa(X \times B) = \nu(B)$ for all $A \in \Sigma_X$ and $B \in \Sigma_Y$. We define $\Pi(\mu,\nu)$ to be the set of all plans with marginals $\mu$ and $\nu$.
\end{definition}

If random variables $X,Y$ have laws $\mu,\nu$, then any coupling of $X,Y$ has a law $\kappa$ which is a plan in $\Pi(\mu,\nu)$.

\begin{definition}\label{def_Markov_operator}
Let $\mu$ and $\nu$ be finite positive measures on $(X,\Sigma_X)$ and $(Y,\Sigma_Y)$ respectively with $||\mu||=||\nu||$, and let $p \in [1,\infty]$. 
We say that a map $T : \mc{L}^p(X,\mu) \to \mc{L}^p(Y,\nu)$ is a \emph{Markov operator}  if:
\begin{enumerate}[label=(\roman*)]
\item $T$ is linear with $T 1_X = 1_Y$;
\item $f \ge 0$ implies $Tf \ge 0$ for all $f \in \mc{L}^p(X,\mu)$;
\item $\int_X f ~d\mu = \int_Y Tf ~d\nu$ for all $f \in \mc{L}^p(X,\mu)$.
\end{enumerate}
\end{definition}

Notice that the definition of Markov operators depends on underlying measures $\mu$ and $\nu$ on $X$ and $Y$ respectively, even when $p=\infty$. We now define some properties for transfunctions that are analogous to (ii) and (iii) from Definition \ref{def_Markov_operator}.

\begin{definition}
Let $\Phi: \mc{M}_X \to \mc{M}_Y$ be a transfunction.
\begin{enumerate}[label=(\roman*)]
\item $\Phi$ is \textit{positive} if $\lambda \ge 0$ implies that $\Phi \lambda \ge 0$ for all $\lambda \in \mc{M}_X$.
\item $\Phi$ is \textit{measure-preserving} if $(\Phi\lambda)(Y)=\lambda(X)$ for all $\lambda \in \mc{M}_X$. 
\item $\Phi$ is \textit{Markov} if it is strongly $\sigma$-additive, positive and measure-preserving. 
\end{enumerate} 
\end{definition}

By \cite{Taylor}, there is a bijective relationship between plans and Markov operators. We will show soon that a relationship between Markov operators and Markov transfunctions exists, which will imply that all three concepts are connected. 

\begin{lemma}\label{lem_linear_bijection}
Let $\mu$ be a finite positive measure on $(X,\Sigma_X)$. Define $J_\mu: \mc{L}^1(X,\mu) \to \mc{M}^1_\mu$ via $J_\mu f = f \mu$. Then $J_\mu$ (hence $J_\mu^{-1}$) is a positive linear isometry.

\end{lemma}
\begin{proof}
Positivity and linearity of integrals with respect to $\mu$ ensure that $J_\mu$ is positive and linear. Surjectivity of $J_\mu$ is the statement of the Radon-Nikodym Theorem. Injectivity and isometry hold because 
$$||J_\mu f|| = ||J_\mu(f^+) - J_\mu(f^-)|| = \int_X f^+ d\mu + \int_X f^- d\mu = \int_X |f| d\mu = ||f||_1.$$
\end{proof}
\begin{theorem}\label{thm_Markov_transfunction_char}
Let $\mu$ and $\nu$ be finite positive measures on $X$ and $Y$ respectively, with $||\mu|| = ||\nu||$ and let $s \in [1,\infty]$. For every Markov operator $T: \mc{L}^s(X,\mu) \to \mc{L}^s(Y,\nu)$, there exists a unique Markov transfunction $\Phi: \mc{M}^s_\mu \to \mc{M}^s_\nu$ such that
$$
\int_B T(1_A) d\nu = \Phi(1_A \mu)(B)
$$ 
for all $A \in \Sigma_X$ and $B \in \Sigma_Y$.

Every Markov transfunction $\Phi: \mc{M}^s_\mu \to \mc{M}^s_\nu$ corresponds to a family of Markov operators $\{T_{\lambda,\rho} : \mc{L}^\infty(X,\lambda) \to \mc{L}^\infty(Y,\rho) ~|~ \lambda \in \mc{M}^{s,+}_\mu ~,~ \rho = \Phi\lambda\}$ which satisfies
$$
\int_B T_{\lambda,\rho}(1_A) d\rho = \Phi(1_A \lambda)(B)
$$
for all $A \in \Sigma_X$ and $B \in \Sigma_Y$.
\end{theorem}
\begin{proof}

We prove the first statements for $s=1$, then extend the argument to other values of $s$. Let $T: \mc{L}^1(X,\mu) \to \mc{L}^1(Y,\nu)$ be a Markov operator. Define $\Phi = J_\nu  ~T~ J_\mu^{-1}$. Since all three operators in the definition of $\Phi$ are positive and strongly $\sigma$-additive, we see that $\Phi$ is also positive and strongly $\sigma$-additive. Next, if $\lambda \in \mc{M}_\mu$, then 
$$
(\Phi \lambda) (Y) = J_\nu (T J_\mu^{-1} \lambda) (Y) = \int_Y T(J_\mu^{-1} \lambda) d\nu = \int_X J_\mu^{-1} (\lambda) d\mu = \lambda(X)
$$
by the definitions of isometries $J_\mu^{-1}$ and $J_\nu$, and by property (iii) of $T$, so $\Phi$ is measure-preserving. Finally, notice that 
$$
\Phi(1_A \mu) (B) = J_\nu T (J_\mu^{-1} (1_A \mu))(B) = J_\nu (T 1_A) (B) = \int_B T(1_A) d\nu
$$ 
for all $A \in \Sigma_X$ and $B \in \Sigma_Y$, hence the relation holds.

Now let $s \in (1,\infty]$ and let $T: \mc{L}^s(X,\mu) \to \mc{L}^s(Y,\nu)$ be a Markov operator. By Theorem 1 from \cite{Taylor}, $T$ can be uniquely extended to a Markov operator $\widehat{T}$ on $\mc{L}^1(X,\mu)$. By our previous argument, $\widehat{T}$ corresponds to a Markov transfunction $\widehat{\Phi}$ defined on $\mc{M}_\mu$. We define $\Phi$ to be the restriction of $\widehat{\Phi}$ to $\mc{M}^s_\mu$. The necessary properties are inherited from the previous argument.

Now we prove the second statement. Let $s \in [1,\infty]$, let $\Phi: \mc{M}^s_\mu \to \mc{M}^s_\nu$ be a Markov transfunction, let $\lambda \in \mc{M}^s_\mu$ be positive, and define $\rho := \Phi(\lambda) \in \mc{M}^s_\nu$, which is also positive. Define $T = T_{\lambda,\rho} := J_\rho^{-1} ~\Phi~ J_\lambda$ with domain $\mc{L}^\infty(X,\lambda)$. Then 
$$T(1_X) = J_\rho^{-1} \Phi (J_\lambda (1_X)) = J_\rho^{-1} (\Phi \lambda) = J_\rho^{-1} \rho = 1_Y.$$ Since all three operators in the definition of $T$ are positive and strongly $\sigma$-additive, we see that $T$ is also positive and strongly $\sigma$-additive, satisfying parts (i) and (ii) of Definition \ref{def_Markov_operator}. Next, if $f \in \mc{L}^\infty(X,\lambda)$, then 
$$
\int_Y T f d\rho = \int_Y J_\rho^{-1} (\Phi J_\lambda f) d\rho = (\Phi(J_\lambda f)) (Y) = (J_\lambda f)(X) = \int_X f d\lambda,
$$
so (iii) of Definition \ref{def_Markov_operator} is met. Finally, notice that 
$$
\int_B T(1_A) d\rho = \int_B J_\rho^{-1} (\Phi J_\lambda (1_A)) d\rho = \Phi(J_\lambda (1_A)) (B) = \Phi(1_A \lambda) (B)
$$
for all $A \in \Sigma_X$ and $B \in \Sigma_Y$, so the relation holds.
 
\end{proof}

One consequence from Theorem \ref{thm_Markov_transfunction_char} is that any Markov transfunction defined on $\mc{M}^s_\mu$ for $s \in [1,\infty]$ uniquely extends or restricts  to $\mc{M}^{s'}_\mu$ for all $s' \in [1,\infty]$, thus the value of $s$ is insignificant. This is analogous to a similar property held by Markov operators, as in \cite{Taylor}.

The remainder of this section aims to emphasize the importance of Theorem \ref{thm_Markov_transfunction_char}. For any $p \in [1,\infty]$, a transfunction $\Phi: \mc{M}^p_\mu \to \mc{M}^p_\nu$, a Markov operator $T: \mc{L}^p(X,\mu) \to \mc{L}^p(Y,\nu)$ and a plan $\kappa \in \Pi(\mu,\nu)$ that satisfy the equalities 
$$
\Phi(1_A \mu)(B) = \int_B T(1_A) ~d\nu = \kappa(A \times B)
$$ 
for all $A\subseteq X$ and $B\subseteq Y$ contain the same information (transportation method), but convey it differently. By extending the equalities above for all $f \in \mc{L}^p(X,\mu)$ and $g \in \mc{L}^q(Y,\nu)$ with $1/p + 1/q = 1$, we have
$$\int_Y g ~d\Phi(f \mu) = \int_Y T(f) ~d(g \nu) = \int_{X \times Y} (f \otimes g) ~d\kappa .$$

Note that if some positive measure $\mu'$ also generates $\mc{M}^p_\mu$, and if we define $\nu'=\Phi(\mu')$, then the same transfunction $\Phi: \mc{M}^p_{\mu} \to \mc{M}^p_{\nu}$ corresponds to a Markov operator $T' : \mc{L}^p(X,\mu') \to \mc{L}^p(Y,\nu')$ and it corresponds to a plan $\kappa'$ with marginals $\mu'$ and $\nu'$. Therefore $T$ and $T'$ are different Markov operators, $\kappa$ and $\kappa'$ are different plans, yet they follow the same ``instructions'' encoded by $\Phi$. In this regard, $\Phi$ is a global way to describe a transportation method independent of marginals. If $\mu'$ instead generates a smaller space than $\mc{M}_\mu$, then $\Phi$ restricted to $\mc{M}_{\mu'}$ contains part but not all of the instructions. Regardless, $\Phi$ will be Markov on this restriction. Notably, if $\mu' = h \mu$, then $\Phi: \mc{M}_{\mu'} \to \mc{M}_{\nu'}$ has associated Markov operator $T_h (f) := T (h f)$ and associated plan $\kappa' = (h \otimes 1_Y) \kappa$.

\section{Radon Adjoints of Transfunctions}

Let $(X,\Sigma_X)$ be a Borel measurable space, let $\mc{F}_X$ be a subset of bounded measurable real-valued functions on $X$ and let $\mc{M}_X$ be a subset of finite signed measures on $X$. Analogously, we have $Y$, $\mc{F}_Y$ and $\mc{M}_Y$. For $f \in \mc{F}_X$ and $\lambda \in \mc{M}_X$, define $\langle f,\lambda \rangle := \int_X f ~d\lambda$. Similarly, for $g \in \mc{F}_Y$ and $\rho \in \mc{M}_Y$, define $\langle g, \rho \rangle := \int_Y g ~d\rho$. Occasionally, the elements within angular brackets shall be written in reverse order without confusion: for example, see Definition \ref{def_biseparate}.

We say that $\{\mc{F}_X , \mc{M}_X\}$ \emph{is a separating pair} if $\langle f_1, \lambda \rangle = \langle f_2 , \lambda \rangle$ for all $\lambda \in \mc{M}_X$ implies that $f_1 = f_2$, and if $\langle f, \lambda_1 \rangle = \langle f, \lambda_2 \rangle$ for all $f \in \mc{F}_X$ implies that $\lambda_1 = \lambda_2$.  In this section, we shall develop some theory for two choices of the collections $\{\mc{F}_X, \mc{M}_X\}$ and $\{\mc{F}_Y, \mc{M}_Y\}$, which we call the continuous setting and the measurable setting.

\begin{definition}\label{def_biseparate}
Let $\{\mc{F}_X,\mc{M}_X\}$ and $\{\mc{F}_Y,\mc{M}_Y\}$ be separating pairs. Let $\Phi: \mc{M}_X \to \mc{M}_Y$ be a transfunction and let $S: \mc{F}_Y \to \mc{F}_X$ be a function. Then $\Phi$ and $S$ are \emph{Radon adjoints} of each other if the equation 
$$\int_Y g ~d\Phi(\lambda) = \int_X S(g) ~d\lambda, ~~~~~ \text{i.e.} ~~~~~ \langle g, \Phi(\lambda) \rangle = \langle S(g) , \lambda \rangle$$ holds for all $g \in \mc{F}_Y$ and $\lambda \in \mc{M}_X$. 
\end{definition}

By utilizing the separation properties of $\langle \cdot,\cdot \rangle$, Radon adjoints of both kinds are unique if they exist. We shall denote the Radon adjoint of $\Phi$ by $\Phi^\ast$ and of $S$ by $S^\ast$.

If $(\Phi, S)$ is a Radon adjoint pair, then 
$$\textstyle \langle g, \Phi \sum_i \lambda_i \rangle = \langle Sg, \sum_i \lambda_i \rangle = \sum_i \langle Sg, \lambda_i \rangle = \sum_i \langle g, \Phi \lambda_i \rangle = \langle g, \sum_i \Phi \lambda_i \rangle ,$$ 
meaning that $\Phi$ is strongly $\sigma$-additive. Similarly,  
$$\textstyle \langle S\sum_i g_i, \lambda \rangle = \langle \sum_i g_i, \Phi \lambda \rangle = \sum_i \langle g_i, \Phi \lambda \rangle = \sum_i \langle Sg_i, \lambda \rangle = \langle \sum_i Sg_i, \lambda\rangle ,$$
meaning that $S$ is linear and uniformly-continuous. 
\begin{example}
If $\Phi = f_\#$ (the push-forward operator) for some measurable $f: X \to Y$, then $\Phi^* (g) = g \circ f = f^\ast g$ (the pull-back operator acting on $g$). This is because $\int_Y g ~d(f_\# \lambda) = \int_X g \circ f ~d\lambda$ for all $g \in \mc{F}_Y$, $\lambda \in \mc{M}_X$. 
\end{example}
\begin{example}
If $X=Y$ and $\Phi \lambda := f \lambda$ for some continuous (or measurable) $f: X \to \mb{R}$, then $\Phi^*(g) = g f$. This is because $\int_X g ~d(f\lambda) = \int_X g f ~d\lambda$ for all $g \in \mc{F}_X$, $\lambda \in \mc{M}_x$.
\end{example}

\begin{definition}\label{def_weak_convergence}
~

\begin{enumerate}[label=(\roman*)]
\item $(f_n)$ \emph{weakly converges} to $f$ in $\mc{F}_X$, notated as $f_n \xrightarrow{w} f$, if every finite regular measure $\lambda$ on $X$ yields $\langle f_n, \lambda \rangle \to \langle f, \lambda \rangle$ as $n \to \infty$. 
\item $(\lambda_n)_{n=1}^\infty$ \emph{weakly converges} to $\lambda$ in $\mc{M}_X$, notated as $\lambda_n \xrightarrow{w} \lambda$, if every bounded continuous $f: X \to \mb{R}$ yields $\langle f, \lambda_n \rangle \to \langle f, \lambda \rangle$ as $n \to \infty$. 
\item An operator $S: \mc{F}_Y \to \mc{F}_X$ is \emph{weakly continuous} if $g_n \xrightarrow{w} g$ in $\mc{F}_Y$ implies that $S g_n \xrightarrow{w} S g$ in $\mc{F}_X$. 
\item A transfunction $\Phi: \mc{M}_X \to \mc{M}_Y$ is \emph{weakly continuous} if $\lambda_n \xrightarrow{w} \lambda$ in $\mc{M}_X$ implies that $\Phi \lambda_n \xrightarrow{w} \Phi \lambda$ in $\mc{M}_Y$.
\end{enumerate} 
\end{definition}

%

Note that weak convergence of $(f_n)$ in Definition \ref{def_weak_convergence} (i) is the same notion as bounded-pointwise convergence.

\section{Approximations of Identity}

For the remainder of this paper, let $X$ and $Y$ be locally-compact Polish spaces, and pick any complete metric for each of them when needed. 

\begin{definition}
For a metric space $(X,d)$ with $x \in X$, $A \subseteq X$ and $\delta > 0$, define $B(x;\delta) := \{z \in X: d(x,z) < \delta\}$ to be the \emph{$\delta$-ball around $x$} and define $B(A;\delta) := \cup_{x \in A} B(x;\delta)$ to be the \emph{$\delta$-inflation around $A$}.
\end{definition}

The following two lemmas aid in showing Proposition \ref{prop_density_measures}.

\begin{lemma}\label{lemma_inflate}
Let $(X,d)$ be a locally compact metric space. The positive function $c: X \to (0,\infty]$ defined via 
$$c(x) := \sup \{\delta > 0 : B(x;\delta) \text{ is precompact}\}$$  
is either identically $\infty$ or it is finite and continuous on $X$. It follows that every compact set $K$ has a precompact inflation $B(K;\delta)$ for some $\delta>0$.
\end{lemma}

%

\begin{lemma}\label{lemma_bubbles}
Let $(X,d)$ be a locally compact Polish metric space. Then there exists a pair of sequences $(x_i)_{i=1}^\infty$ from $X$ and $(\beta_i)_{i=1}^\infty$ from $(0,1]$ and there exists a function $p: \mb{N} \to \mb{N}$ such that for all $n \in \mb{N}$, 

$$K_n := \bigcup_{i=1}^{p(n)} \overline{B(x_i, \beta_i / n)}$$ is compact with $K_{n+1} \supseteq K_{n+1}^\circ \supseteq K_n$ and $\cup_{n=1}^\infty K_n = X$.
\end{lemma}

Using the setup from Lemma \ref{lemma_bubbles}, we define the collection of sets $$C_{n,i} := \overline{B(x_i, \beta_i / n)} - \cup_{j<i} \overline{B(x_j, \beta_j / n)}$$ for all $n,i \in \mb{N}$. It follows for any $n \in \mb{N}$ that $\cup_{i=1}^{p(n)} C_{n,i} = K_n$. 
\begin{definition}
A measure $\mu$ is called a \emph{point-mass measure at $x$} if $\mu(A) = 1$ when $x \in A$ and $\mu(A) = 0$ when $x \notin A$. A finite linear combination of point-mass measures is called a \emph{simple measure}. 
\end{definition}

It is straightforward to show that simple measures are regular. The following proposition suggests a method to create approximations of identity, which shall be discussed in their respective sections below.

\begin{proposition}\label{prop_density_measures}
Simple measures on a second-countable locally compact Hausdorff space form a dense subset of all finite regular measures with respect to weak convergence. 
\end{proposition}

\begin{proof}

Construct sequences $(x_i)_{i=1}^\infty$, $(\beta_i)_{i=1}^\infty$, $p: \mb{N} \to \mb{N}$ and $(C_{n,i})$ via Lemma \ref{lemma_bubbles}. Fix some positive measure $\lambda \in \mc{M}^+_X$. Construct a sequence $(\lambda_n)_{n=1}^\infty$ of positive simple measures via $\lambda_n := \sum_{i=1}^{p(n)} \lambda(C_{n,i}) \delta_{x_i}$. We will show that $\lambda_n \xrightarrow{w} \lambda$. In doing so, we fix some function $f \in \mc{C}_b(X)$ and show that $\langle f, \lambda_n \rangle \to \langle f, \lambda \rangle$. For density of signed measures, one utilizes the Jordan decomposition and applies a similar argument for each component.

Let $\eps > 0$. Define $\eta := \eps / (3||f|| + 3||\lambda|| + 1)$ so that $||f|| ~ \eta < \eps/3$ and that $||\lambda|| ~ \eta < \eps/3$. Choose some natural $M$ such that $\lambda(K_M^c) < \eta$. Apply Lemma \ref{lemma_inflate} to obtain some $\alpha > 0$ with $L := \overline{B(K_M;\alpha)}$ being compact. By uniform continuity of $f|_L$, choose some natural $N > M$ such that $2/N < \alpha$ and for all $x \in L$, $f(B(x;2/N) \cap L) \subseteq B(f(x); \eta).$

Now let $n > N$. Define $\rho_{n,M} := \sum_{i=1}^{p(n)} \lambda(C_{n,i} \cap K_M) \delta_{x_i}$. Notice that $C_{n,i} \cap K_M \not= \varnothing$ implies that $x_i \in B(K_M;1/n)$ and that $ C_{n,i} \subseteq  B(K_M;2/n) \subseteq L$, resulting in $f(C_{n,i}) \subseteq B(f(x_i);\eta)$. Three observations can be made: \begin{enumerate}
\item[(a)] $|\langle f, \lambda - 1_{K_M} \lambda\rangle| \le ||f|| \cdot \lambda(K_M^c) < ||f|| ~ \eta;$ 
\item[(b)] $|\langle f, 1_{K_M} \lambda - \rho_{n,M} \rangle| \le \left| \int_{K_M} f ~d\lambda - \sum_{i=1}^{p(n)} f(x_i) \lambda(C_{n,i} \cap K_M)\right| < ||\lambda||~\eta;$ 
\item[(c)] $|\langle f, \rho_{n,M} - \lambda_n \rangle| \le ||f||  \sum_{i=1}^{p(n)} \lambda(C_{n,i} \cap K_M^c) \le ||f||~\lambda(K_M^c) < ||f|| ~ \eta.$
\end{enumerate}
Therefore, $|\langle f, \lambda - \lambda_n \rangle| < 3 (\eps/3) = \eps$. 
\end{proof}

For any finite signed measure $\lambda$ on $X$, the sequence $(\lambda_n)$ of simple measures from Proposition \ref{prop_density_measures} weakly converges to $\lambda$, hence the sequence of transfunctions $(I_n)$ given by $I_n: \lambda \mapsto \lambda_n = \sum_{i=1}^{p(n)} \langle 1_{C_{n,i}} , \lambda \rangle \delta_{x_i}$ is an approximation of identity. 

The approximation of identity above is simply described with characteristic functions $(1_{C_{n,i}})$ and point-mass measures $(\delta_{x_i})$. However, in each of the two settings below, either the characteristic functions must be replaced by bounded continuous functions or the point-mass measures must be replaced by compactly-supported measures that are absolutely continuous with respect to some underlying measure. With the correct choice of replacements, the same argument as given in Proposition \ref{prop_density_measures} can be applied, yielding valid approximations of identities for the respective settings.

\section{Results in the Continuous Setting:\\ $\mc{F}_X = \mc{C}_b(X) ~,~ \mc{M}_X = \mc{M}_{fr}(X)$}
Let $\mc{F}_X = \mc{C}_b(X)$ denote the Banach space of all bounded continuous functions on $X$ with the uniform norm and let $\mc{M}_X=\mc{M}_{fr}(X)$ denote the Banach space of all finite (hence, regular) signed measures on $X$ with the total variation norm. Develop $Y$, $\mc{F}_Y$, and $\mc{M}_Y$ analogously. It is known that $\{\mc{F}_X, \mc{M}_X\}$ is a separating pair in this setting.

An approximation of identity can be formed in this setting: keep the point-mass measures $\rho_{n,i} := \delta_{x_i}$, then for each natural $n$, replace the characteristic functions  $\{1_{C_{n,i}} : 1 \le i \le p(n)\}$ used in Proposition \ref{prop_density_measures} with positive compactly supported continuous functions $\{f_{n,i} : 1 \le i \le p(n)\}$ such that $f_{n,i} \le 1_{B(C_{n,i}; 1/n)}$ and that $1_{K_n} \le \sum_{i=1}^{p(n)} f_{n,i} \le 1_{B(K_n; 1/n)}$. Then an approximation of identity in the continuous setting is given by the sequence $(I_n)$, where 
$$I_n: \lambda \mapsto\sum_{i=1}^{p(n)} \langle f_{n,i} , \lambda \rangle ~\rho_{n,i} = \sum_{i=1}^{p(n)} \langle f_{n,i} , \lambda \rangle ~\delta_{x_i}.$$

\begin{theorem}\label{thm_Radon_adjoint_char}
Every strongly $\sigma$-additive and weakly-continuous transfunction $\Phi: \mc{M}_{fr}(X) \to \mc{M}_{fr}(Y)$ has a strongly $\sigma$-additive and weakly-continuous Radon adjoint $S: \mc{C}_b(Y) \to \mc{C}_b(X)$. Conversely, every strongly $\sigma$-additive and weakly-continuous operator $S: \mc{C}_b(Y) \to \mc{C}_b(X)$ has a strongly $\sigma$-additive and weakly-continuous Radon adjoint $\Phi: \mc{M}_{fr}(X) \to \mc{M}_{fr}(Y)$. When the Radon adjoint pair exists, their operator norms equal (with respect to total-variation and uniform-convergence).
\end{theorem}

\begin{proof}
For the first claim, define $S(g)(x) := \langle g, \Phi(\delta_x) \rangle$ for all $g \in \mc{F}_Y$, $x \in X$. It follows that $\langle S(g), \delta_x \rangle = \langle g, \Phi(\delta_x) \rangle$ for all $g,x$. Let $x_n \to x$ on $X$, so that $\delta_{x_n} \xrightarrow{w} \delta_x$, which means that $\Phi(\delta_{x_n}) \xrightarrow{w} \Phi(\delta_x)$. Also let $g_n \to g$ bounded-pointwise in $\mc{F}_Y$ (i.e. $g_n \xrightarrow{w} g$). Then the statements below ensure that $S(g) \in \mc{F}_Y$, that $S$ is bounded (hence uniform-continuous) and that $S$ is bounded-pointwise-continuous (via the Dominated Convergence Theorem):
$$S(g)(x_n) = \langle S(g), \delta_{x_n} \rangle = \langle g, \Phi(\delta_{x_n}) \rangle \to \langle g, \Phi(\delta_x) \rangle = \langle S(g), \delta_x \rangle = S(g)(x);$$
$$||S(g)|| = \sup_{x \in X} |S(g)(x)| = \sup_{x \in X} |\langle S(g), \delta_x \rangle| = \sup_{x \in X}|\langle g, \Phi(\delta_x)\rangle| \le ||g|| \cdot ||\Phi||;$$
$$S(g_n)(x) = \langle S(g_n),\delta_x\rangle = \langle g_n, \Phi \delta_x \rangle \to \langle g,\Phi \delta_x \rangle = \langle S(g), \delta_x \rangle = S(g)(x).$$

Since countable linear combinations of point-mass measures are weakly dense in $\mc{M}_X$, the linearity and weak-continuity of the second coordinate in the $\langle \cdot,\cdot \rangle$ structure and the weak-continuity of $\Phi$ yields that $\langle S(g), \lambda \rangle = \langle g, \Phi \lambda \rangle $ for all $g \in \mc{F}_Y$ and $\lambda \in \mc{M}_X$. Hence, $S$ is the Radon adjoint of $\Phi$ with the desired properties.

For the second claim, note that for every $\lambda \in \mc{M}_X$, the continuous functional $g \mapsto \langle S(g), \lambda \rangle$ defined on $\mc{C}_0(Y)$ has Riesz representation $\langle \cdot, \Phi(\lambda) \rangle$ for some unique signed measure $\Phi(\lambda) \in \mc{M}_Y$. Defining $\Phi$ in this manner for all $\lambda$, we obtain the equation $\langle S(g), \lambda \rangle = \langle g, \Phi \lambda \rangle$ for all $g \in \mc{C}_0(Y)$ and $\lambda \in \mc{M}_X$. $C_0(Y)$ is dense in $C_b(Y)$ with respect to bounded-pointwise convergence, so with $\mc{C}_0(Y) \ni g_n \xrightarrow{w} g \in \mc{C}_b(Y)$, it follows that $\langle g_n, \Phi \lambda \rangle \to \langle g, \Phi \lambda \rangle$ by the Dominated Convergence Theorem. Similarly, bounded-pointwise-continuity of $S$ ensures that $S(g_n) \xrightarrow{w} S(g)$, which means that $\langle S(g_n), \lambda \rangle \to \langle S(g), \lambda \rangle$ by the Dominated Convergence Theorem. Therefore, $\langle S(g), \lambda \rangle = \langle g, \Phi \lambda \rangle$ for all $g \in \mc{F}_Y$ and $\lambda \in \mc{M}_X$, implying that $\Phi$ is the Radon adjoint of $S$. An earlier remark shows that $\Phi$ is strongly $\sigma$-additive. To see that $\Phi$ is weakly-continuous, let $\lambda_n \xrightarrow{w} \lambda$. Then $\langle g, \Phi \lambda_n \rangle = \langle S(g), \lambda_n \rangle \to \langle S(g), \lambda \rangle = \langle g, \Phi\lambda \rangle$. Therefore, $\Phi \lambda_n \xrightarrow{w} \Phi \lambda$. Finally, $||\Phi|| \le ||S||$, hence $||\Phi|| = ||S||$, follows via
$$||\Phi \lambda|| = \sup_{||g||=1} |\langle g, \Phi \lambda \rangle| = \sup_{||g||=1} |\langle S(g), \lambda \rangle| \le ||S||\cdot ||\lambda||.$$
\end{proof}

\section{Results in the Measurable Setting: \\ $\mc{F}_X = \mc{L}^\infty (X,\mu) ~,~ \mc{M}_X = \mc{M}^\infty_\mu$}

In this setting, let $(X,\Sigma_X,\mu)$ be a finite measure space, let $\mc{F}_X := \mc{L}^\infty(X,\mu)$ and let $\mc{M}_X := \mc{M}^\infty_\mu$. Define $(Y,\Sigma_Y,\nu), \mc{F}_Y, \mc{M}_Y$ analogously. Then it is straightforward to verify that $\mc{F}_X$ and $\mc{M}_X$ separate each other.

An approximation of identity can be formed in this setting: for each natural $n$ and $1 \le i \le p(n)$, replace each point-mass measure $\delta_{x_i}$ used in Proposition \ref{prop_density_measures} with the measure $\rho_{n,i} := 1_{C_{n,i}} \mu$ and define $f_{n,i} := 1_{C_{n,i}}/\mu(C_{n,i})$ when $\mu(C_{n,i})>0$; otherwise, define $f_{n,i} = 0$.
That is, an approximation of identity in the measurable setting is given by the sequence $(I_n)$, where 
$$I_n: \lambda \mapsto \sum_{i=1}^{p(n)} \langle f_{n,i} , \lambda \rangle ~ \rho_{n,i} = \displaystyle \sum_{\substack{i=1 \\ \mu(C_{n,i}) > 0}}^{p(n)} \left\langle \dfrac{1_{C_{n,i}}}{\mu(C_{n,i})} , \lambda \right\rangle ~ 1_{C_{n,i}} \mu.$$ 
The following lemma will be used in the proof of the next theorem:
\begin{lemma}\label{lem_adj_Markov}
For every strongly $\sigma$-additive transfunction $\Phi: \mc{M}^\infty_\mu \to \mc{M}^\infty_\nu$, there is a unique strongly $\sigma$-additive transfunction $\Phi^\dag: \mc{M}^\infty_\nu \to \mc{M}^\infty_\mu$ such that $\Phi^\dag(1_B \nu) (A) = \Phi(1_A \mu) (B)$ for all $A \subseteq X, B \subseteq Y$, which implies that $\langle f, \Phi^\dag (g \nu) \rangle = \langle \Phi(f \mu), g \rangle$ for all $f \in \mc{L}^\infty (X,\mu)$, $g \in \mc{L}^\infty(Y,\nu)$. Also, $\Phi^{\dag \dag} = \Phi$.
\end{lemma}
\begin{proof}
Let $\Phi$ be strongly $\sigma$-additive. For fixed $B \subseteq Y$, it follows by strong $\sigma$-additivity of $\Phi$ that the set function $ A \mapsto \Phi(1_A \mu)(B)$
is a measure. Define this measure to be $\Psi(1_B \nu)$. Then $\Psi$, defined on $\{1_B \nu ~|~ B \subseteq Y\}$ is a strongly $\sigma$-additive transfunction that behaves like $\Phi^\dag$ in the equality above. $\Psi$ can be linearly extended to $\mc{M}^\infty_\nu$ according to the following equalities for $A\subseteq X, g \cong \sum_j \beta_j 1_{B_j} $ with $\sum_j |\beta_j| < \infty$:

\begin{align*}
\Psi (g \nu) (A) = \sum_{j=1}^\infty \beta_j \Psi(1_{B_j} \nu) (A) = \sum_{j=1}^\infty \beta_j \Phi (1_A \mu) (B_j) = \int_Y g ~d\Phi (1_A \mu).
\end{align*}
The extended $\Psi$ is strongly $\sigma$-additive on $\mc{M}_\nu^\infty$. A similar calculation shows that 
$$\int_X f ~d\Psi(g\nu) = \int_Y g ~d\Phi(f\mu)$$ 
for all $f \in \mc{L}^\infty(X,\mu)$ and $g \in \mc{L}^\infty(Y,\nu)$. Therefore $\Phi^\dag$ is uniquely determined to be $\Psi$. Finally, $\Phi^{\dag\dag}(1_A \mu)(B) = \Phi^\dag(1_B \nu)(A) = \Phi(1_A \mu)(B)$ for all $A\subseteq X$ and $B\subseteq Y$, so $\Phi^{\dag\dag} = \Phi$.
\end{proof} 

If $\Phi$ is Markov, then $\Phi^\dag$ is also Markov. Furthermore, the plans $\kappa, \kappa^\dag$ corresponding to $\Phi,\Phi^\dag$ respectively are dual to each other: that is, $\kappa(A \times B) = \kappa^\dag(B \times A)$ for all measurable sets $A \subseteq X, B \subseteq Y$. However, $\Phi^\dag$ is sensitive to the choice of measures $\mu$ and $\nu$, which is not ideal when working with non-injective extensions of Markov transfunction $\Phi$.

\begin{theorem}\label{thm_Radon_2}
Every strongly $\sigma$-additive $\Phi: \mc{M}^\infty_\mu \to \mc{M}^\infty_\nu$ has a linear and bounded Radon adjoint $S: \mc{L}^\infty(Y,\nu) \to \mc{L}^\infty(X,\mu)$. Conversely, every linear and bounded operator $S: \mc{L}^\infty(Y,\nu) \to \mc{L}^\infty(X,\mu)$ has a strongly $\sigma$-additive adjoint $\Phi: \mc{M}^\infty_\mu \to \mc{M}^\infty_\nu$.
\end{theorem}

\begin{proof}

Assume that $\Phi: \mc{M}^\infty_\mu \to \mc{M}^\infty_\nu$ is strongly $\sigma$-additive and define $S := J_\mu^{-1} \Phi^\dag  J_\nu$ with domain $\mc{L}^\infty (Y,\nu)$. Then $S$ is linear and bounded. $S = \Phi^\ast$ follows because for any $f \in \mc{L}^\infty (X,\mu)$ and $g \in \mc{L}^\infty(Y,\nu)$,
$$\langle g, \Phi (f \mu)\rangle = \langle \Phi^\dag (g \nu), f \rangle = \langle (J_\mu  S) g , f \rangle = \langle S(g), f \mu \rangle .$$ 

On the other hand, assume that $S : \mc{L}^\infty(Y,\nu) \to \mc{L}^\infty(X,\mu)$ is linear and bounded. Then define $\Psi := J_\mu  S  J_\nu^{-1}$ and $\Phi:= \Psi^\dag$. Then $\Phi$ is strongly $\sigma$-additive. $\Phi = S^\ast$ follows because for any $f \in \mc{L}^\infty (X,\mu)$ and $g \in \mc{L}^\infty(Y,\nu)$,
$$\langle S(g), f \mu \rangle = \langle (J_\mu^{-1}  \Psi) (g\nu) , f\mu \rangle = \langle \Phi^\dag (g \nu), f \rangle = \langle g, \Phi (f \mu)\rangle.$$ 
\end{proof}

\section{Simple Transfunctions}

Let $\mc{F}_X$, $\mc{F}_Y$, $\mc{M}_X$, and $\mc{M}_Y$ be defined in either the continuous setting or the measurable setting.

\begin{definition}
A transfunction $\Phi: \mc{M}_X \to \mc{M}_Y$ is \emph{simple} if there exist functions $(f_i)_{i=1}^m$ from $\mc{F}_X$ and there exist measures $(\rho_i)_{i=1}^m$ from $\mc{M}_Y$ such that

$$\forall \lambda \in \mc{M}_X ,~~~ \Phi \lambda = \sum_{i=1}^m \langle f_i, \lambda \rangle \rho_i.$$

\end{definition}

It is straightforward to verify that simple transfunctions are strongly $\sigma$-additive. In the continuous setting, simple transfunctions are also weakly-continuous. Therefore by Theorem \ref{thm_Radon_adjoint_char} in the continuous setting or Theorem \ref{thm_Radon_2} in the measurable setting, the Radon adjoint $\Phi^\ast$ exists and satisfies
$$\forall g\in \mc{F}_Y,~~~ \Phi^\ast g=\sum_{i=1}^m \langle g, \rho_i \rangle f_i.$$

Note that the approximations of identity covered in the continuous setting in (Subsection 3.2) and in the measurable setting in (Subsection 3.3) involve sequences of simple transfunctions.

\begin{theorem}
In either continuous or measurable settings, linear weakly-continuous transfunctions can be approximated by simple transfunctions with respect to weak convergence; that is, simple transfunctions form a dense subset of linear weakly-continuous transfunctions with respect to weak convergence.
\end{theorem}

\begin{proof}
Let $\Phi: \mc{M}_X \to \mc{M}_Y$ be weakly-continuous transfunction and fix ${\lambda \in \mc{M}_X}$. Define $\Phi_n := \Phi ~ I_n$, where $I_n: \lambda \mapsto \sum_{i=1}^{p(n)} \langle f_{n,i} , \lambda \rangle \rho_{n,i} $ forms the approximation of identity as defined in either Subsections 3.2 (continuous setting) or 3.3 (measurable setting). Then $\Phi_n \lambda = \sum_{i=1}^{p(n)} \langle f_{n,i} , \lambda \rangle \Phi \rho_{n,i}$, implying that $\Phi_n$ is a simple transfunction. It follows by $I_n \lambda \xrightarrow{w} \lambda$ and by weak-continuity of $\Phi$ that $\Phi_n \lambda = \Phi (I_n \lambda) \xrightarrow{w} \Phi \lambda$.
\end{proof}

\section{Applications: Optimal Transport}

Markov transfunctions provide a new perspective to optimal transport theory.

\begin{definition}\label{def_c_optimal}
Let $(X,\Sigma_X, \mu)$ and $(Y, \Sigma_Y, \nu)$ be Polish measure spaces with finite positive measures $\mu$ and $\nu$,  respectively, with $||\mu||=||\nu||$.
A \emph{cost function} is any continuous function $c: X \times Y \to [0,\infty)$. A plan $\kappa \in \Pi(\mu,\nu)$ is \emph{$c$-optimal} if $\int_{X \times Y} c ~d\kappa \leq \int_{X\times Y} c ~d\pi$ for all $\pi \in \Pi(\mu,\nu)$.
A Markov transfunction $\Phi: \mc{M}_X \to \mc{M}_Y$ is \emph{$c$-optimal on $\mu$} if the corresponding plan $\kappa$ with marginals $\mu$ and $\Phi \mu$ is $c$-optimal, and $\Phi$ is simply \emph{$c$-optimal} if it is $c$-optimal on $\mc{M}_X$. 
\end{definition}

The next proposition implies that optimal inputs for $\Phi$ form a large class of measures.

\begin{proposition}
Let $(X,\Sigma_X), (Y, \Sigma_Y)$ be Polish spaces, let $c$ be a cost function, and let $\Phi: \mc{M}_X \to \mc{M}_Y$ be a $c$-optimal Markov transfunction. If $\Phi$ is $c$-optimal on $\mu \in \mc{M}_X$, then $\Phi$ is $c$-optimal on $\mc{M}^\infty_\mu$.
\end{proposition}
\begin{proof}
The proof follows easily from Theorem 4.6 in \cite{Villani} on the inheritance of optimality of plans by restriction.
\end{proof}

In the next theorem, we provide a ``warehouse strategy" which approximates the optimal cost between fixed marginals with respect to some cost function. In summary, the input marginal is first subdivided by local regions, and the subdivided measures are sent to point mass measures -- warehouses -- within their respective regions. Second, mass is transferred between warehouses via the discrete transport problem. Finally, the warehouses locally redistribute to form the output marginal. The overall cost of transport via the warehouse strategy approaches the optimal cost as the size of the regions decreases.

\begin{theorem}\label{thm_approx_opt_cost}
Let $(X, \Sigma_X)$ be a locally compact Polish measurable space with complete metric $d$, let $\lambda$ and $\rho$ be finite positive compactly-supported measures with $||\lambda|| = ||\rho||$, and let $c: X \times X \to [0,\infty)$ be a cost function with $c(x,y) \le \alpha d(x,y)^p$ for some constants $\alpha,p > 0$.
The optimal cost between marginals $\lambda, \rho$ with respect to $c$ can be sufficiently approximated by the costs of simple Markov transfunctions.
\end{theorem}

\begin{proof}
Assume the continuous setting and consider the approximation of identity $(I_n)$ from Subsection 3.2. For large $n$, we create a composition of three simple Markov transfunctions: $\lambda$ first maps to $I_n \lambda = \sum_{i=1}^{p(n)} \langle f_{n,i}, \lambda \rangle \delta_{x_i}$, which maps to $I_n \rho = \sum_{i=1}^{p(n)} \langle f_{n,i}, \rho \rangle \delta_{x_i}$, which finally maps to $\rho$. These steps are measure-preserving because $K_n$ (from Lemma \ref{lemma_bubbles}) contains the supports of $\lambda$ and $\rho$ for large $n$. The most crucial goal is to determine the optimal simple Markov transfunction for the middle step.

The first and last steps cost no more than $\alpha n^{-p} ||\lambda||$ each, which reduces to 0 as $n \to \infty$. This means that the optimal cost between marginals $\lambda_n$ and $\rho_n$ approaches the optimal cost between marginals $\lambda$ and $\rho$ as $n \to \infty$. Solving the former optimal cost is the well-known discrete version of the Monge-Kantorovich transport problem. 

By approximating each of the values $\langle f_{n,i}, \lambda \rangle \approx a_{n,i}/z$ and $\langle f_{n,i}, \rho \rangle \approx b_{n,i}/z$ for natural numbers $a_{n,i}, b_{n,i}, z$ with $1 \le i \le p(n)$, the middle step can approximately be interpreted as the Assignment Problem on a weighted bipartite graph between vertex sets $P$ and $Q$, where $P$ denotes a set created by forming $a_{n,i}$ copies of a vertex corresponding to each $\delta_{x_i}$ in $\lambda_n$, $Q$ denotes the set created by forming $b_{n,j}$ copies of a vertex corresponding to each $\delta_{x_j}$ in $\rho_n$, and drawing edges between these vertices with weight $c(x_i, x_j)$. This problem has been studied, and can be solved in polynomial time of $|P| = \sum_{i=1}^{p(n)} a_{n,i} \approx ||\lambda|| z$; the Hungarian method is one well-known algorithm \cite{Kuhn}.
\end{proof}

Although Theorem \ref{thm_approx_opt_cost} provides a sequence of simple transfunctions that approximate the optimal cost between fixed marginals, the sequence is not expected to converge weakly to an optimal Markov transfunction, as the solutions to the middle step could vary greatly as $n$ increases. Consequently, we can find a Markov transfunction whose cost between marginals is sufficiently close to the optimal cost, but Theorem \ref{thm_approx_opt_cost} does not provide an optimal Markov transfunction.

However, for any Markov transfunction between fixed marginals, the next theorem yields an approximation by simple Markov transfunctions with respect to weak convergence. Consequently, the cost between the marginals of the constructed sequence of simple Markov transfunctions approaches the cost for the original transfunction.

\begin{theorem}\label{thm_approx_opt_transfunction}
Let $(X, \Sigma_X, \mu)$ and $(Y, \Sigma_Y, \nu)$ be locally compact Polish measure spaces with finite compactly-supported positive measures $\mu$ and $\nu$ such that $||\mu|| = ||\nu||$. Any Markov transfunction $\Phi: \mc{M}_\mu \to \mc{M}_\nu$ can be approximated by simple Markov transfunctions with respect to weak convergence.
\end{theorem}
\begin{proof}
Assume the measurable setting and consider the approximation of identity $(I_n)$ from Subsection 3.3 with respect to $\mu$. Let $n$ be large so that $K_n$ (from Lemma \ref{lemma_bubbles}) contains the supports of $\mu$ and $\nu$. 

Let $\kappa$ be the plan corresponding to Markov transfunction $\Phi$ from Theorem \ref{thm_Markov_transfunction_char}. For $1 \le i,j \le p(n)$, the quantity $\kappa(C_{n,i} \times C_{n,j})$ represents how much mass transfers from $1_{C_{n,i}} \mu$ to $1_{C_{n,j}} \nu$. If $\mu(C_{n,i}) \nu(C_{n,j}) > 0$, then we can approximate nonzero measure  $(1_{C_{n,i}} \otimes 1_{C_{n,j}}) \kappa$ with 
$$\kappa_{n,i,j} := \kappa(C_{n,i} \times C_{n,j}) \dfrac{1_{C_{n,i}} \mu}{\mu(C_{n,i})} \times \dfrac{1_{C_{n,j}} \nu}{\nu(C_{n,j})}.
$$
Otherwise, we define $\kappa_{n,i,j} := 0$. Then $\kappa_n := \sum_i \sum_j \kappa_{n,i,j}$ is a plan from $\Pi(\mu, \nu)$ which corresponds to a Markov transfunction $\Phi_n$ from Theorem \ref{thm_Markov_transfunction_char}. 

Next, we show that $\kappa_n \xrightarrow{w} \kappa$ as $n \to \infty$. Let $c \in \mc{C}_b(X\times Y)$ with $||c||\le 1$, and for $1 \le i,j \le p(n)$, let $\beta_{n,i,j} := \sup c(C_{n,i} \times C_{n,j}) - \inf c(C_{n,i} \times C_{n,j})$. By uniform continuity of $c$ on $K_n \times K_n$, we have that $\beta_n := \displaystyle \max\{\beta_{n,i,j} | 1 \le i,j \le p(n)\}  \to 0$ as $n \to \infty$, which implies that
$$|\langle c, \kappa - \kappa_n\rangle| \le \sum_i \sum_j \beta_{n,i,j} ~\kappa(C_{n,i} \times C_{n,j}) \le \beta_n ||\kappa|| \to 0.
$$

There are some properties of $\Phi_n$ worth noting: $\Phi_n$ maps $\mc{M}_\mu^\infty$ to  $\text{span}\{1_{C_{n,j}} \nu\}$; $\Phi_n$ behaves as a matrix when applied to $\text{span}\{1_{C_{n,i}} \mu\}$; the structure of $\kappa_n$ guarantees that $\Phi_n = \Phi_n I_n$. If we choose bases $(1_{C_{n,i}} \mu)$ and $(1_{C_{n,j}} \nu)$, the matrix $M_n$ representing $\Phi_n$ has entries $M_n (j,i) := \kappa(C_{n,i} \times C_{n,j}) / \nu(C_{n,j})$.

Let $\lambda \in \mc{M}_\mu^\infty$ and for $1 \le i \le p(n)$. Then
\begin{align*}
\Phi_n \lambda &= \Phi_n I_n \lambda = \sum_j \left\langle \sum_i \dfrac{\kappa(C_{n,i} \times C_{n,j})}{\nu(C_{n,j})} \dfrac{1_{C_{n,i}}}{\mu(C_{n,i})}, \lambda \right\rangle 1_{C_{n,j}} \nu \\
&=\sum_j \left\langle \sum_i ||\kappa_{n,i,j}|| 1_{C_{n,i}}, \lambda \right\rangle 1_{C_{n,j}} \nu,
\end{align*}
showing that $\Phi_n$ is simple.

We now show that $\Phi_n \lambda \xrightarrow{w} \Phi \lambda$ as $n \to \infty$. Let $g \in \mc{C}_b(Y)$ with $||g|| \le 1$. Let $\eps > 0$. Since $\lambda = f \mu$ for some $f \in \mc{L}^\infty(X,\mu)$, choose some $\tilde{f} \in \mc{C}_b(X)$ such that $||(f - \tilde{f}) \mu|| < \eps/3$. Since $||\Phi^\ast|| = ||\Phi_n^\ast||=1$, we have that 
$$|\langle g, \Phi  (f - \tilde{f}) \mu \rangle| = |\langle \Phi^\ast g, (f - \tilde{f}) \mu \rangle| \le ||\Phi^\ast g|| \cdot ||(f - \tilde{f}) \mu||  \le \eps/3,$$
and that
$$|\langle g, \Phi_n  (f - \tilde{f}) \mu \rangle| = |\langle \Phi_n^\ast g, (f - \tilde{f}) \mu \rangle| \le ||\Phi_n^\ast g|| \cdot ||(f - \tilde{f}) \mu|| \le \eps/3.$$
Since $\kappa_n \xrightarrow{w} \kappa$ as $n \to \infty$ and $\tilde{f} \otimes g \in \mc{C}_b (X \times Y)$, there is some natural $N$ so that for all $n \ge N$,
$$|\langle g, (\Phi - \Phi_n) \tilde{f} \mu \rangle| = |\langle \tilde{f} \otimes g , \kappa - \kappa_n \rangle| < \eps/3.$$
It follows from above that $|\langle g, (\Phi - \Phi_n) \lambda \rangle| \le \eps$ for $n \ge N$ by the triangle inequality. 
\end{proof}
Theorem \ref{thm_approx_opt_transfunction} can be strengthened by removing the assumptions that $\mu$ and $\nu$ are compactly supported; the approximation of identity may not capture all of $\mu$ nor $\lambda$, and $\kappa_n \in \Pi(1_{K_n}\mu, 1_{K_n}\nu)$ may not belong to $\Pi(\mu, \nu)$, but the rest of the analysis holds. Notably, to show $\kappa_n \xrightarrow{w} \kappa$ as $n \to \infty$, the inequalities become
$$|\langle c, \kappa - \kappa_n\rangle| \le \kappa(K_n^c) + \sum_i \sum_j \beta_{n,i,j} ~\kappa(C_{n,i} \times C_{n,j}) \le \kappa(K_n^c) + \beta_n ||\kappa_n|| \to 0.
$$

\end{document}